\documentclass[11pt]{amsart}
\usepackage{amssymb,amsfonts,amsthm,amsmath,mathrsfs,xspace,hyperref}
\usepackage{graphicx}
\usepackage[francais,english]{babel}
\usepackage{inputenc}
\usepackage[T1]{fontenc}
\usepackage[centering]{geometry}
\usepackage{csquotes}
\usepackage{color}
\usepackage{enumerate}
\usepackage{mathabx}
\usepackage{stmaryrd}
\usepackage{enumitem}
\usepackage{comment}

  \newcommand{\R}{\ensuremath{\mathbb{R}}}%
  \newcommand{\Z}{\ensuremath{\mathbb{Z}}}%
  \newcommand{\N}{\ensuremath{\mathbb{N}}}%
	\newcommand{\F}{\ensuremath{\mathcal{F}}}%
		\newcommand{\Gcal}{\ensuremath{\mathcal{G}}}%

    \newcommand{\A}{\ensuremath{\mathcal{A}}}%

  \newcommand{\Sym}{\ensuremath{\operatorname{Sym}}}%
      \newcommand{\Alt}{\ensuremath{\operatorname{Alt}}}%
    \newcommand{\alt}{\ensuremath{\operatorname{Alt}}}%
        \newcommand{\vol}{\ensuremath{\operatorname{vol}}}%

  \newcommand{\supp}{\ensuremath{\textrm{supp}}}%
    \newcommand{\acts}{\ensuremath{\curvearrowright}}%
  \newcommand{\sub}{\ensuremath{\operatorname{Sub}}}%
  \newcommand{\homeo}{\ensuremath{\operatorname{Homeo}}}%
    \newcommand{\aut}{\ensuremath{\operatorname{Aut}}}%
                \newcommand{\diam}{\ensuremath{\operatorname{diam}}}%

			\newcommand{\Fsf}{\ensuremath{\mathsf{F}}}
			\newcommand{\Asf}{\ensuremath{\mathsf{A}}}

\newcommand{\Class}{\ensuremath{\mathcal{C}}}
\newcommand{\Clin}{\ensuremath{\mathcal{C}(\mathrm{lin})}}

\theoremstyle{definition}
  \newtheorem{defin}{Definition}[section]

\theoremstyle{plain}
  \newtheorem{thm}[defin]{Theorem}
  \newtheorem{thm-intro}[defin]{Theorem}
  \newtheorem{prop}[defin]{Proposition}
    \newtheorem{prop-intro}[defin]{Proposition}
    \newtheorem{prop-def}[defin]{Proposition-Definition}

  \newtheorem{cor}[defin]{Corollary}
   \newtheorem{lem}[defin]{Lemma}

\theoremstyle{remark}
  \newtheorem{remark}[defin]{Remark}

  \newtheorem{example}[defin]{Example}
  
\begin{document}

  \date{November 22, 2024}	
	
\title{On the growth of actions of free products}

\author{Adrien Le Boudec}
\address{CNRS, UMPA - ENS Lyon, 46 all\'ee d'Italie, 69364 Lyon, France}
\email{adrien.le-boudec@ens-lyon.fr}

\author{Nicol\'as Matte Bon}
\address{
	CNRS,
	Institut Camille Jordan (ICJ, UMR CNRS 5208),
	Universit\'e Claude Bernard Lyon 1,
	43 blvd.\ du 11 novembre 1918,	69622 Villeurbanne,	France}

\email{mattebon@math.univ-lyon1.fr}

\author{Ville Salo}
\address{University of Turku, Department of Mathematics and Statistics, Finland.}

\email{vosalo@utu.fi}

\thanks{Supported by the LABEX MILYON (ANR-10-LABX-0070) of Universite de Lyon, within the program "Investissements d'Avenir" (ANR-11-IDEX-0007) operated by the French National Research Agency.}

\maketitle

%
%

\begin{abstract}
	If $G$ is a finitely generated group and $X$ a $G$-set, the growth of the action of $G$ on $X$ is the function that measures the largest cardinality of a ball of radius $n$ in the (possibly non-connected) Schreier graph $\Gamma(G,X)$. We consider the following stability problem: if $G,H$ are finitely generated groups admitting a faithful action of growth bounded above by a function $f$, does the free product $G \ast H$ also admit a faithful action of growth bounded above by $f$? We show that the answer is positive under additional assumptions, and negative in general. In the negative direction, our counter-examples are obtained with $G$ either the commutator subgroup of the topological full group of a minimal and expansive homeomorphism of the Cantor space, or $G$ a Houghton group. In both cases, the group $G$ admits a faithful action of linear growth, and we show that $G\ast H$  admits no faithful action of subquadratic growth provided $H$ is non-trivial. {In the positive direction, we describe a class of groups that admit actions of linear growth and is closed under free products and exhibit examples within this class, among which the Grigorchuk group.} \end{abstract}

\begin{flushright}
\begin{minipage}[t]{0.85\linewidth}\itshape\small
Dedicated with admiration to Slava Grigorchuk on the occasion of his 70th birthday.
\end{minipage}
\end{flushright}



\section{Introduction}

Let $G$ be a finitely generated group, equipped with some  finite symmetric generating set $S$ (implicit in what follows). For a $G$-set $X$, we denote by $\Gamma(G, X)$ the Schreier graph of the action, with vertex set $X$  and edges $(x, sx)$ for $x\in X$ and $s\in S$. Note that we do not assume the action of $G$ on $X$ to be transitive, so that the graph $\Gamma(G, X)$ need not be connected. Let $\vol_{G, X}\colon \N \to \N$ be the function that measures the volume of the largest ball appearing in $\Gamma(G, X)$:
\[\vol_{G, X}(n)=\max_{x\in X} |B_G(n) \cdot x|,\]
where $B_G(n)$ denotes the ball of radius $n$ around the identity in $G$ with respect to the word metric associated to $S$. 
For two functions $f_1, f_2\colon \N \to \N$, we write $f_1(n)\preccurlyeq f_2(n)$ if there exists $C>0$ such that $f_1(n)\le Cf_2(Cn)$, and $f_1(n)\simeq f_2(n)$ if $f_1(n)\preccurlyeq f_2(n) \preccurlyeq  f_1(n)$. The function $\vol_{G, X}(n)$ does not depend on the choice of  the generating set $S$ up to the equivalence relation $\simeq$. 

\begin{defin} \label{defin-class}
Given a function $f\colon \N\to\N$, we denote by $\Class(f)$ the class of finitely generated groups $G$ such that there exists a faithful $G$-set $X$ such that $\vol_{G, X}(n) \preccurlyeq f(n)$.
\end{defin}

It is not hard to see that the class $\Class(f)$ is stable under taking finitely generated subgroups, and under passing to a finite index overgroup \cite[\S 1.3]{LB-MB-solv}. We use the special notation $\Clin$ for the distinguished case $f(n)\simeq n$, which is the slowest possible growth for a faithful action of an infinite group. The class $\Clin$ is richer than it might look at first sight. It is not hard to check that it contains all virtually abelian groups. The fact that non-abelian free groups belong to $\Clin$ is well-known and can be traced back to Schreier \cite{Sch}.  A famous example of a group in $\Clin$ is the Grigorchuk group (and all Grigorchuk groups $G_\omega$ from \cite{Gri-growth}),  see Bartholdi and Grigorchuk \cite{Bar-Gri-Hecke}. Many other interesting groups of dynamical origin are naturally given by an action of linear growth (and hence belong to $\Clin$): topological full groups of homeomorphisms of the Cantor set, and Nekrashevych's fragmentation of dihedral groups \cite{Nek-frag}.  More examples of groups in $\Clin$ include the lamplighter group $C_p\wr \Z$, the Houghton groups, or Neumann's groups from \cite{Neumann-manygroups} (the latter are defined by an action whose orbits are all finite, yet having linear growth). 

In \cite{Salo-graph-prod} the third author investigated the class of subgroups of the topological full group of full shifts over $\Z$ (which is contained in $\Clin$), and gave sufficient conditions under which graph products of such groups remain inside that class \cite[Theorem 3]{Salo-graph-prod}. That result implies in particular that the class $\Clin$ contains all right angled Artin groups  \cite[Theorem 1]{Salo-graph-prod}, and hence all their subgroups as well. Combined with results of many authors, this is a vast source of examples of groups in $\Clin$, as many groups are known to embed (upon to passing to a finite index subgroup) in some right angled Artin group. For example, this is the case for instance surface groups, and more generally any hyperbolic group acting geometrically on a CAT(0) cube complex, as follows from Agol's final step in the solution of the virtual Haken conjecture \cite{Agol} combined with the previous results of Haglund and Wise \cite{Ha-Wi1}.

Motivated by this last class of examples, we consider here the question whether the class $\Clin$ (or the classes $\Class(f)$ for more general $f$) is closed under free products, or more generally graph products.  There is a natural strategy to attempt to answer that question affirmatively. Given two groups $G_1$ and $G_2$ and $X_i$ a $G_i$-set, $i=1,2$, one may consider any set $X=X_1\sqcup_{A_1\cong_{\phi} A_2} X_2$ obtained by identifying suitable subsets $A_1 \subset X_1$ and $A_2 \subset X_2$ via a bijection $\phi\colon A_1\to A_2$. The action of each $G_i$ extends to $X$ as a trivial action outside of the copy of $X_i$ in $X$. This defines an action of the free product $G=G_1\ast G_2$. On the one hand, if the subsets along which the $X_i$ are glued are chosen wisely, one can hope to control the growth function $\vol_{G, X}(n)$ in terms of $\vol_{G_i, X_i}(n), i=1,2$. On the other hand, if the gluing is sufficiently generic, the action of $G$ on $X$ is likely to be faithful. The question therefore becomes whether, for $G_1, G_2$  in $\Class(f)$, the tension between these two conditions can be conciliated.  A naive implementation of this strategy (suitably adapted to more general graph products) gives the following. We refer to \S \ref{subsec-def-graph-prod} for the definition of graph products. 




 \begin{prop-intro}[Proposition \ref{prop-largebound-graph-prod}] \label{p-naive-bound}
Let $k\geq 1$ and $G_1,\ldots,G_k \in \Class(f)$. Then the free product $G=G_1\ast \cdots \ast G_k$ belongs to $\Class(g)$, where $g(n)  = n f(n)$. More generally, any graph product of the $G_i$'s belongs to $\Class(g)$. 
\end{prop-intro}

Proposition \ref{p-naive-bound} is very often not optimal. In many cases, appropriate choices  allow to show that a graph product of groups in $\Class(f)$ remains in $\Class(f)$.  In \S \ref{subsec-subclass-stable} we discuss a sufficient condition for this to be the case. That condition is an elaboration of a condition considered in \cite{Salo-graph-prod}.  Proposition \ref{prop-large-displ-stable} provides in particular a direct and elementary proof of the fact that right-angled Artin groups belong to $\Clin$ (which essentially follows the same lines as the proof in \cite{Salo-graph-prod}, but avoids the technicalities coming from the full group setting). More generally, Proposition \ref{prop-large-displ-stable} allows to exhibit examples of graph products in the class $\Clin$:

\begin{prop-intro} \label{prop-intro-grig}
For $k\ge 1$, the free product $G_1\ast \cdots \ast G_k$ belongs to $\Clin$ whenever each free factor $G_i$ is isomorphic to one of the following:
\begin{itemize}
\item a finite group;
\item a finitely generated abelian group;
\item a finitely generated subgroup of a right-angled Artin group;
\item the Grigorchuk group;
\item the lamplighter group $(\Z/p\Z) \wr \Z$ for $p\ge 2$. 

\end{itemize}
More generally, the graph product of any finite family of groups in the list above belongs to $\Clin$.
\end{prop-intro}

\subsection{The main results}
 
The main goal of this paper is to provide examples that show that the bound obtained from the simple construction from Proposition \ref{p-naive-bound} is sharp in general, even for free products. We consider two families of examples. 

The first are topological full groups of minimal group actions on the Cantor set, more precisely their alternating subgroups in the sense of Nekrashevych \cite{Nek-simple} (see \S  \ref{subsec-full-group} for definitions).  Following \cite{LB-MB-solv}, we shall say that a finitely generated group $G$ has a \textbf{Schreier growth gap} $f(n)$ if every faithful $G$-set $X$ satisfies $\vol_{G, X}(n)\succcurlyeq f(n)$.

\begin{thm-intro} \label{t-intro-full}
Let $\Gcal \acts  \mathcal{X}$ be a minimal expansive action of a finitely generated group on the Cantor set, and set $f(n)=\vol_{\Gcal, X}(n)$. Let $G$ be the alternating full group of the action (so that $G$ is finitely generated and belongs to $\Class(f(n))$, see \S \ref{subsec-full-group}). Then for every non-trivial finitely generated group $H$, the group $G\ast H$ has a Schreier growth gap $nf(n)$. 
\end{thm-intro} 



A relevant special case in the previous theorem is $\mathcal{G}=\Z$ (in that case $G$ coincides with the commutator subgroup of the topological full group, by the results of Matui \cite{Mat-simple}). In that case we obtain examples of groups $G\in \Clin$ such that $G\ast H$ has a Schreier growth gap $n^2$ for any non-trivial group $H$. The quadratic gap is optimal in general, as Proposition \ref{p-naive-bound}  says that $G\ast H$ has a faithful action of growth $\simeq n^2$ for instance when $H$ is cyclic and non-trivial.

Our second family of examples is the family of Houghton groups $H_r, r\ge 2$. Recall that the group $H_2$ is defined as the group of permutations of $\Z$ that coincide with a translation outside a finite set; see \S \ref{subsec-houghton} for the definition of $H_r$. Each group $H_r$ is finitely generated and belongs to $\Clin$.

\begin{thm-intro} \label{t-intro-Houghton}
Let $G=H_r$ be the Houghton group on $r\ge 2$ rays. Then  for every non-trivial finitely generated group $H$, the group $G\ast H$ has a Schreier growth gap $n^2$. 
\end{thm-intro}

\subsection{Outline of the proof of Theorems \ref{t-intro-full} and \ref{t-intro-Houghton}}

These two results share the same proof mechanism. It relies on the description of the confined subgroups of the groups $G$ in their statements, obtained in previous works of the first two authors. Recall that a subgroup $H$ of a group $G$ is confined if the set of conjugates of $H$ does not accumulate on the trivial subgroup in the space $\sub(G)$ of subgroups of $G$, endowed with the Chabauty topology. When $G$ is finitely generated, convergence in the space $\sub(G)$ can be interpreted in the space of marked Schreier graphs on the corresponding coset spaces $G/H$. The study of confined subgroups of a group $G$ leads (among other applications) to results on the growth function $\vol_{G, X}(n)$ of  $G$-actions, such as the existence of a Schreier growth gap \cite{MB-full,LB-MB-comm-lemma,LB-MB-solv}. For $G$ as in Theorem \ref{t-intro-full} or \ref{t-intro-Houghton}, a classification of the confined subgroups of $G$ has been obtained respectively in  \cite{MB-full} and \cite{LB-MB-solv}. Using these results, we show that every faithful action of $G$ whose growth is close to the growth of the natural defining action of $G$ (respectively on the Cantor set or on the bouquet of rays), must be ``almost'' conjugate to it. A common feature of the two situations is that the group $G$ contains elements whose support is a very sparse subset of  the Schreier graph of the natural action. We exploit this fact and the previous result about actions of small growth of $G$ to show that in the graph for any faithful action of $G\ast H$, there must be short jumps between regions that are far in the graph of the restricted action of $G$. This is that phenomenon that forces an additional factor $n$ in the growth. The detailed proofs are given in Section \ref{sec-gap-free-prod}.

It is worth comparing the case of the Grigorchuk group in Proposition \ref{prop-intro-grig} with Theorem \ref{t-intro-full}. The Grigorchuk group and topological full groups share various common features: in particular they appear through a micro-supported action on a compact space, a condition which plays an important role in the study of confined subgroups (indeed the confined subgroups of the Grigorchuk group are also understood \cite{LB-MB-comm-lemma}). Also Theorem \ref{t-intro-full} is applicable to many fragmentations of dihedral groups in the sense of Nekrashevych \cite{Nek-frag} (a family of groups to which the Grigorchuk group also belongs). Despite these similarities, here they exhibit an opposite behaviour. The main reason for this difference is the fact that the action of the Grigorchuk group on the vertices of the binary tree satisfies a condition that we call orbits of controlled diameter (Definition \ref{defin-m-control}), which means that every non-trivial element of the group must move some vertex of the tree by a distance comparable to the total diameter of the Schreier graph of the action on the corresponding level of the tree (Proposition \ref{p-Grigorchuk-linear-displacement}). Although the proof of this fact is not difficult, it relies on rather specific features of the Grigorchuk group. 

\subsection*{Acknowledgements} We thank the referees for their careful reading of the paper.


%

\bigskip

%
%


%


\section{Stability results}

\subsection{Graph products} \label{subsec-def-graph-prod}

Let $I$ be a set, and let $\Delta$ denote the diagonal in $I^2$. Let $c$ be a map $I^2 \setminus \Delta \to \left\lbrace 0,1 \right\rbrace$ such that $c(i,j) = c(j,i)$ for all $i \neq j$. Note the data of $c$ is equivalent to the data of an undirected graph with vertex set $I$ with no loop and simple edges. If $(G_i)_{i \in I}$ a family of groups indexed by $I$, the \textbf{graph product} $P =\mathsf{GP}((G_i)_{I}; c)$ of the family $(G_i)_{i \in I}$ associated to $c$ is the quotient of the free product $\Asterisk_I G_i$ by the relations $[G_i,G_j]$ for all $i,j$ such that $c(i,j)=1$. We say that $P$ is a \textbf{finite graph product} if $I$ is finite.

%

\subsection{Preliminaries}


\begin{defin}
	Let $q \geq 1$.	Let $(X_1,\ldots,X_q) = (X_k)_{k \leq q}$ be a $q$-tuple of graphs, and for every $k \leq q$ let $(e_k,s_k)$ be two distinct points of $X_k$. The \textbf{gluing} of the $q$-tuple $(X_k,e_k,s_k)_{k \leq q}$ is the graph obtained by taking the disjoint union $X_1 \sqcup \cdots \sqcup X_q$, and identifying $s_k$ and $e_{k+1}$ for every $1 \leq k \leq q-1$. It is denoted $\mathcal{G}((X_k,e_k,s_k)_{k \leq q})$. 
\end{defin}

(The points $e_1$ and $s_q$ do not play any role, but we keep two based points in  $X_1$ and $X_q$ as well in order to simplify notation.) Note that for every $k \leq q$ the map from $X_k$ to  $\mathcal{G}((X_k,e_k,s_k)_{k \leq q})$ is a graph isomorphism onto its image. In the sequel we identify $X_k$ with its image in $\mathcal{G}((X_k,e_k,s_k)_{k \leq q})$. Note also that the condition $e_k \neq s_k$ ensures that the images of $X_k$ and $X_{k+r}$ in $\mathcal{G}((X_k,e_k,s_k)_{k \leq q})$ are disjoint for all $r \geq 2$.


\begin{lem} \label{lem-growth-gluing}
	Let $(X_k)_{k \leq q}$ be a $q$-tuple of graphs, and suppose that every ball of radius $n \geq 1$ in $X_k$ has cardinality at most $f(n)$ for every $k \leq q$. Then every ball of radius $n$ in $\mathcal{G}((X_k,e_k,s_k)_{k \leq q})$ has cardinality at most $(2n+1) \cdot f(n)$. 
\end{lem}

\begin{proof}
	Every ball $B$ of radius $n$ in $\mathcal{G}((X_k,e_k,s_k)_{k \leq q})$  intersects at most $2n+1$ members of $(X_k)_{k \leq q}$, and is covered by $2n+1$ balls of radius at most $n$ within these $X_k$. The statement follows.
\end{proof}



In the sequel we fix a set $I$, a function $c:I^2 \setminus \Delta \to \left\lbrace 0,1 \right\rbrace$ such that $c(i,j) = c(j,i)$ for all $i \neq j$, and a family of graphs $\F$ such that $\F$ admits a partition indexed by $I$, with blocks denoted $\F_i$, $i \in I$. For $X \in \F$, we write $i(X)$ for the unique element of $I$ such that $X \in \F_{i(X)}$. 

\begin{defin} \label{defin-c-admissible}
A $q$-tuple $(X_k)_{k \leq q}$ of elements of $\F$ is \textbf{$c$-admissible} if $i(X_k) \neq i(X_{k+1})$ and $c(i(X_k),i(X_{k+1})) = 0$ for every $k=1,\dots,q-1$. By extension we also say that $(X_k,e_k,s_k)_{k \leq q}$ is $c$-admissible if $(X_k)_{k \leq q}$ is $c$-admissible.
\end{defin}

\begin{defin}
	We let $\mathcal{G}(\F,c)$ be the disjoint union of all $\mathcal{G}((X_k,e_k,s_k)_{k \leq q})$, where $q \geq 1$ is any positive integer, $(X_k)_{k \leq q}$ is any $c$-admissible $q$-tuple, and $(e_k,s_k)_{k \leq q}$ is any sequence such that $e_k,s_k$ are distinct elements of $X_k$ for every $k \leq q$.
\end{defin}


Now suppose that $(G_i)_{i \in I}$ is a family of groups, and for every $i \in I$ there is an action of $G_i$ on $X$ for every $X \in \F_i$. Let $P = \mathsf{GP}((G_i)_{I}; c)$ be the graph product of $(G_i)_I$ associated to $c$. For every $c$-admissible $(X_k,e_k,s_k)_{k \leq q}$, there is a natural $G_i$-action on $\mathcal{G}((X_k,e_k,s_k)_{k \leq q})$ that extends the $G_i$-action on each $X_k$ such that $i(X_k) = i$, that is defined by declaring that $G_i$ acts trivially on $X_{k} \setminus \left\lbrace e_{k},s_k \right\rbrace $ for every $k$ such that $i(X_k) \neq i$. This is well defined because of the property that no two consecutive $X_k$ can be in $\F_i$, guaranteed by the assumption $i(X_k) \neq i(X_{k+1})$ in Definition \ref{defin-c-admissible}. This defines an action of the free product $\Asterisk_I G_i$ on $\mathcal{G}((X_k,e_k,s_k)_{k \leq q})$. The fact that $(X_k,e_k,s_k)_{k \leq q}$ is $c$-admissible ensures that if $c(i,j)=1$, then the groups $G_i$ and $G_j$ act on $\mathcal{G}((X_k,e_k,s_k)_{k \leq q})$ with disjoint support. Hence the action of $\Asterisk_I G_i$ on $\mathcal{G}((X_k,e_k,s_k)_{k \leq q})$ factors through the graph product $P = \mathsf{GP}((G_i)_{I}; c)$.

\begin{remark} \label{rmk-graph-free-prod}
	Suppose for every $i \in I$ the group $G_i$ is generated by a subset $S_i$, and that the graph structure on each $X \in \F_i$ is the Schreier graph of the $G_i$-action on $X$. Then the graph $\mathcal{G}((X_k,e_k,s_k))$  is essentially the Schreier graph of the action of $P = \mathsf{GP}((G_i)_{I}; c)$ associated to the generating subset $\bigcup_I S_i$. The only difference is that loops should be added at the places where the action of the $G_i$'s has been extended in a trivial way.
\end{remark}

\begin{lem} \label{lem-faithful-gluing}
	Suppose that for every $i \in I$, the $G_i$-action on $\bigsqcup_{X \in \F_i} X$ is faithful. Then the action of the graph product $P = \mathsf{GP}((G_i)_{I}; c)$ on $\mathcal{G}(\F,c)$ is faithful. 
\end{lem}

\begin{proof}
	Let $g$ be a non-trivial element of $P$. Consider the decompositions of the form \[ g = g_n \cdots g_{1},  \] where for each $s$ we have that $g_{s}$ is non-trivial and there is a (necessarily unique) $i_s \in I$ such that $g_{s} \in G_{i_s}$. Among all such decompositions, we choose one such that $n$ is minimal. Since $g$ is non-trivial, $n \geq 1$. We define a sequence $(r_1,\ldots,r_q)$ inductively by setting $r_1 = 1$, and defining $r_{k+1}$ as the smallest $s \geq r_k +1$ such that $c(i(s),i(r_k)) = 0$ (equivalently, $G_{i(s)}$ does not commute with $G_{i(r_k)}$). We also define $r_{q+1} = n+1$. Minimality of $n$ implies that for every $1 \leq k \leq q$ and every $r_k + 1 \leq s \leq r_{k+1}-1$, we have $g_s \notin G_{i(r_k)}$. Let $\gamma_k = g_{r_{k+1}-1} \cdots g_{r_k}$ for $1 \leq k \leq q$, so that $g = \gamma_q \cdots \gamma_1$.  
	
	Since $G_i$ acts faithfully on $\bigsqcup_{X \in \F_i} X$ for every $i \in I$, for every $k \leq q$ one can find $X_k \in \F_{i_{r_k}}$ and  $x_k \in X_{k}$ such that $g_{r_k}(x_k) \neq x_k$. We set $e_{k} = x_k$, $s_{k} = g_{r_k}(x_k)$. We look at how the element $g$ acts on $\mathcal{G}((X_k,e_k,s_k)_{k \leq q})$. Since $g_s$ does not belong to  $G_{i(r_k)}$ for every $r_k + 1 \leq s \leq r_{k+1}-1$, $g_s$ acts trivially on $s_k$. Hence $\gamma_k(e_k) = s_k$ for every $k \leq q$, and hence it follows that $g(e_1) = s_q$. So if $q=1$ then $g$ acts non-trivially since $e_1 \neq s_1$. If $q \geq 2$ then $X_q$ and $X_1$ are either disjoint (when $q \geq 3$), or intersect only along $s_1 = e_2$ if $q=2$. Hence in all cases the element $g$ moves $e_1$. So for every non-trivial element $g$ of $P$, there exists a $\mathcal{G}((X_k,e_k,s_k)_{k \leq q})$ on which $g$ acts non-trivially. So $P$ acts faithfully on $\mathcal{G}(\F,c)$.
\end{proof}

Recall that the class $\Class(f)$ has been defined in the introduction (Definition \ref{defin-class}). 

\begin{prop} \label{prop-largebound-graph-prod}
	Let $f \colon \N\to \N$, and let $(G_i)_I$ be a finite collection of groups in the class $\Class(f)$. Then any graph product $P = \mathsf{GP}((G_i)_{I}; c)$ belongs to $\Class(g)$ with $g(n) = n f(n)$.
\end{prop}

\begin{proof}
For each $i \in I$ there is a $G_i$-set $X_i$ such that the $G_i$-action on $X_i$ is faithful and verifies $\vol_{G_i, X_i}(n) \preccurlyeq f(n)$. The statement then follows from Lemma \ref{lem-growth-gluing} and Lemma \ref{lem-faithful-gluing} (and Remark \ref{rmk-graph-free-prod}), applied to  $\F_i = \left\lbrace X_i \right\rbrace $ and $\F =  \bigsqcup_I \F_I$.
\end{proof}

\subsection{A subclass closed under graph products} \label{subsec-subclass-stable}

In this section we prove that, under a mild assumption on the function $f$, a certain subclass of $\Class(f)$ is closed under finite graph products (Proposition \ref{prop-large-displ-stable}). The definition of this subclass is inspired by the work of the third author \cite{Salo-graph-prod}, and Proposition \ref{prop-large-displ-stable} elaborates on \cite[Theorem 3]{Salo-graph-prod}. In particular Proposition \ref{prop-large-displ-stable} provides an elementary proof of the result from \cite{Salo-graph-prod} that finitely generated right-angled Artin groups admit a faithful action with linear growth. (Actually, Theorem 3 in \cite{Salo-graph-prod} is stronger than that).

For an increasing function $f: \R_+ \to \R_+$, consider the following condition: \begin{equation} \label{partition} \exists C_1 > 0 \, ; \, \forall k \geq 1 \, \, \forall \rho_1,\ldots,\rho_k \geq 0, \, \sum_{i \leq k} f(\rho_i) \leq C_1 f \left(C_1 \sum_{i \leq k} \rho_i \right). \end{equation}

It is a rather mild condition. A sufficient condition for \eqref{partition} to hold is that $n \mapsto f(n) /n$ is increasing. For instance  $f(n) = n^\alpha$ satisfies \eqref{partition} for every $\alpha \ge 1$.


\begin{lem} \label{lem-growth-m-control}
	Let $(X_k)_{k \leq q}$ be a $q$-tuple of finite graphs, and $e_k,s_k$ distinct elements of $X_k$ for every $k \leq q$. Suppose that there is an increasing function $f: \R_+ \to \R_+$ and $C > 0$ such that: \begin{enumerate}
		\item \label{growth-Xk} every ball of radius $n$ in $X_k$ has cardinality at most $C f(Cn)$ for every $k \leq q$.
		\item $f$ satisfies \eqref{partition}.
		\item \label{diam-Xk} $\mathrm{diam}(X_k) \leq C d(e_k,s_k)$ for every $k \leq q$.
	\end{enumerate}
	Then there is a constant $C'> 0$ depending only on $C$ and $f$ such that every ball of radius $n$ in $\mathcal{G}((X_k,e_k,s_k)_{k \leq q})$ has cardinality at most $C'  f(C' n)$. 
\end{lem}

\begin{proof}
	Let $C_1$ be as in \eqref{partition}. Let $n \geq 1$. Let $x$ be a vertex in $\mathcal{G}((X_k,e_k,s_k)_{k \leq q})$, and let $\ell$ be such that $x$ belongs to $X_\ell$. Let $B(n)$ denote the ball of radius $n$ around $x$ in $\mathcal{G}((X_k,e_k,s_k)_{k \leq q})$. Let $X_{\ell - \ell_1}, \ldots X_\ell, \ldots X_{\ell + \ell_2}$ be the members of $(X_k)_{k \leq q}$ that $B(n)$ intersects. Certainly we have \[ d(e_{\ell+1},s_{\ell+1}) +   \cdots + d(e_{\ell + \ell_2 - 1},s_{\ell + \ell_2 - 1}) \leq n  \] and \[ d(e_{\ell-1},s_{\ell-1}) +   \cdots + d(e_{\ell - \ell_1 + 1},s_{\ell - \ell_1 + 1}) \leq n,  \] and $B(n)$ is contained in the union of $X_{\ell - \ell_1+1}, \ldots, X_{\ell-1}, X_{\ell+1}, \ldots, X_{\ell + \ell_2-1}$ and the intersection between $B(n)$ and $X_\ell \cup X_{\ell - \ell_1} \cup X_{\ell + \ell_2}$: \[ |B(n)| \leq \sum_{k=1}^{\ell_1-1} |X_{\ell-k}| + \sum_{k=1}^{\ell_2-1} |X_{\ell+k}| + |B_{X_{\ell}}(x,n)|  + |B_{X_{\ell - \ell_1}}(s_{\ell - \ell_1},n)| + |B_{X_{\ell + \ell_2}}(e_{\ell + \ell_2},n)|. \]  Conditions  \ref{growth-Xk} and \ref{diam-Xk} ensure $|X_k| \leq C f(C \mathrm{diam}(X_k)) \leq C f(C^2d(e_k,s_k))$ for every $k \leq q$. Hence we obtain \begin{align*}  |B(n)| & \leq C \sum_{k=1}^{\ell_1-1} f(C^2 d(e_{\ell-k},s_{\ell-k})) + C \sum_{k=1}^{\ell_2-1} f(C^2d(e_{\ell+k},s_{\ell+k})) + 3Cf(Cn) \\ & \leq C_1 C f \left(C_1 C^2 \sum_{k \leq \ell_1-1} d(e_{\ell-k},s_{\ell-k}) \right) +  C_1 C f \left(C_1 C^2 \sum_{k \leq \ell_2-1} d(e_{\ell+k},s_{\ell+k}) \right) + 3Cf(Cn) \\ & \leq 2C_1 C f(C_1 C^2n) + 3Cf(Cn),\\
	\end{align*}
	where in the second inequality we have used that $f$ satisfies \eqref{partition}, and the third inequality follows from the above upper bounds $\sum d(e_{\ell-k},s_{\ell-k}), \sum d(e_{\ell+k},s_{\ell+k}) \leq n $.
\end{proof}



\begin{defin} \label{defin-m-control}
	Let $G$ be a finitely generated group, and $X$ a $G$-set. We say that the $G$-action on $X$ has \textbf{orbits of controlled diameter} if all the $G$-orbits in $X$ are finite, and if there is $C > 0$ such that for every non-trivial element $g \in G$, there is a $G$-orbit $O$ in $X$ and $x \in O$ such that $gx \neq x$ and $\mathrm{diam}(\Gamma(G,O)) \leq C d(x,gx)$. 
\end{defin}

Here by $\Gamma(G,O)$ we mean the Schreier graph of the $G$-action on $O$ (which is a connected component in the larger graph $\Gamma(G,X)$), with respect to a finite symmetric generating set $S$. It is easily checked that this condition does not depend on the choice of $S$ (although the constant $C$ can depend on $S$). Note that the definition forces in particular the group $G$ to be residually finite.

\begin{defin}
	We denote by $\Class_{\mathrm{CD}}(f) \subset \Class(f) $ the class of groups admitting an action with orbits of controlled diameter and of growth at most $f(n)$.
\end{defin}

\begin{example} \label{e-LD} \noindent
	\begin{enumerate}
		\item Every faithful action of a finite group has orbits of controlled diameter.
		\item The action of $\Z$ on $\bigsqcup_{n \geq 1} \Z / n \Z$  has orbits of controlled diameter. Hence $\Z$ belongs to $\Class_{\mathrm{CD}}(\mathrm{lin})$.
		\item If $G_1, \ldots, G_k$ admit actions with orbits of controlled diameter respectively on $X_1, \ldots, X_k$, then the action of $G_1 \times \cdots \times G_k$ on $X_1 \sqcup \cdots \sqcup X_k$ has orbits of controlled diameter. Hence each class $\Class_{\mathrm{CD}}(f)$ is closed under direct products. 
	\end{enumerate}
\end{example}

Here is a slightly more elaborated example:

\begin{prop} \label{prop-wreath-LD}
For every $d \ge 1$ and $p \geq 2$, the wreath product $G=(\Z/p\Z)\wr \Z^d$  belongs to $\Class_{\mathrm{CD}}(f)$ for $f(n)=n^d$. 
\end{prop}

\begin{proof}
	We denote elements of $G$ as pairs $(f, u)$, with  $f\colon \Z^d\to \Z/p\Z$ of finite support and $u\in \Z^d$. For $m\ge 2$, set $H_m=\{(f, u) \colon \quad \sum_{v\in (m\Z)^d} f(v)=0, u\in (m\Z)^d\}$. Then $H_m$ is a finite index subgroup of $G$, of index $pm^d$. Set $X_m:=G/H_m$ and $X=\sqcup_m  X_m$. We check that the action of $G$ on $X$ has orbits of controlled diameter and satisfies $\vol_{G, X}(n)\simeq n^d$. Consider the standard generating set $S=\{(\delta_0, 0), (0, e_i), i=1, \dots d\}$, where $\delta_0\colon \Z^d\to \Z/p\Z$ takes the value 1 on $0$ and 0 elsewhere, and $e_1, \ldots, e_d$ is the standard basis of $\Z^d$. Note first that the map $G\to (\Z/p\Z)\times (\Z/m\Z)^d$ that sends $(f, u)\in G$  to $(\sum_{v\in m\Z^d} f(v), u \mod m)$ descends to a bijection of the coset space $G/H_m$ with $(\Z/p\Z)\times (\Z/m\Z)^d$, so that we may identify $X_m$ with $(\Z/p\Z)\times (\Z/m\Z)^d$. Under this identification, the action of $G$ on $X_m$ coincides with the standard wreath product action of the natural quotient $(\Z/p\Z) \wr (\Z/m\Z)^d$. Explicitly, the action of elements in the generating set $S$ is described as follows: the lamp generator $(\delta_0 , 0)$ permutes cyclically $(\Z/p\Z)\times \{0\}$ and acts trivially elsewhere; each element $(0, e_i)$ maps $(r, u \text{ mod } m)$ to $(r, u+e_i \text{ mod } m)$.  In particular the Schreier graph of the $G$-action on $X_m$ is isomorphic (ignoring loops) to the graph obtained taking  a cycle of length $p$ and gluing to each  point a copy of the standard Cayley graph of $(\Z/ m\Z)^d$.  It follows that $\vol_{G, X}(n)\simeq n^d$, and that $C_1m \le \mathrm{diam}(X_m)\le C_2 m$ for some constants $C_1, C_2$. To check that the action has orbits of controlled diameter, let $g=(f, v)$ be a non-trivial element of $G$. Suppose first that $v\neq 0$, and choose $m=2|v|$, where $|\cdot|$ denotes the standard word metric of $\Z^d$. Then, since the natural projection of $X_m$ to $(\Z/m\Z)^d$ is equivariant, for every $x\in X_m$ we have $d(x, gx) \ge |v|\ge \frac{1}{2}C_1\mathrm{diam}(X_m)$. Suppose now that $v=0$. Let $u$ be such that $f(u)\neq 0$ and with $|u|$ maximal, and choose again $m>2|u|$. Let $x\in X_m\simeq (\Z/p\Z)\times (\Z/m\Z)^d$ be the point $x=(0, u \text{ mod } m)$. Then $gx=(f(u),  u \text{ mod } m)$.   From the description of the graph $\Gamma(G, X_m)$ one can see that $d(x, gx)\ge 2|u| \ge C_1 \mathrm{diam}(X_m)$. 
\end{proof}


Our main motivation for considering the class $\Class_{\mathrm{CD}}(f)$ is the following:

\begin{prop} \label{prop-large-displ-stable}
	If $f: \R_+ \to \R_+$ is increasing and satisfies \eqref{partition}, then the class $\Class_{\mathrm{CD}}(f)$ is closed under finite graph products. 
\end{prop}

\begin{proof}
	Let $(G_i)_I$ be a finite collection of groups in $\Class_{\mathrm{CD}}(f)$. For every $i \in I$ let $Y_i$ be a $G_i$-set such that the $G_i$-action on $Y_i$ has orbits of controlled diameter and growth at most $f(n)$. Let $S_i$ be a finite generating subset of $G_i$ and $C_i$ a constant as in Definition \ref{defin-m-control}. Let $P = \mathsf{GP}((G_i)_{I}; c)$ be a graph product, generated by $S = \bigcup S_i$. Set $C = \max_i C_i$. 
	
	For $i \in I$, let $\F_i$ be the set of $G_i$-orbits in $Y_i$, and $\F$ the disjoint union of the $\F_i$. We consider the subset of $\mathcal{G}(\F,c)$, denoted by $Y$, consisting of the disjoint union of all $\mathcal{G}((X_k,e_k,s_k)_{k \leq q})$, where $q \geq 1$, $(X_k)_{k \leq q}$ is a $c$-admissible $q$-tuple of elements of $\F$, and $e_k,s_k$ are distinct elements of $X_k$ with the condition that  $\mathrm{diam}(X_k)\leq C d(e_k,s_k)$ for every $k \leq q$. The group $P$ acts on $Y$, and every $P$-orbit in $Y$ is finite because all the members of $\F$ are finite. Moreover Lemma \ref{lem-growth-m-control} ensures $\vol_{P, Y}(n) \preccurlyeq f(n)$. Hence to conclude the proof it suffices to see that the $P$-action on $Y$ has orbits of controlled diameter.
	
	Let $g$ be a non-trivial element of $P$. We repeat verbatim the argument in the proof of Lemma  \ref{lem-faithful-gluing}. We find a decomposition $g = \gamma_q \cdots \gamma_1$ and a $c$-admissible $(X_k,e_k,s_k)_{k \leq q}$ such that $\gamma_k(e_k) = s_k$ for every $k \leq q$ (so that $g(e_1) = s_q$). The assumption that the $G_i$-action on $Y_i$ has orbits of controlled diameter allows to ensure that $\mathrm{diam}(X_k) \leq C d(e_k,s_k)$ for every $k \leq q$. Hence if we write $X = \mathcal{G}((X_k,e_k,s_k)_{k \leq q})$, then $X \in Y$ and  \[ \mathrm{diam}(X) \leq \sum_{k=1}^q \mathrm{diam}(X_k) \leq C  \sum_{k=1}^qd(e_k,s_k) = C d(e_1,s_q). \] \qedhere
\end{proof}

\subsection{The Grigorchuk group} \label{subsec-Grigorchuk-group}


In this section we prove that the Grigorchuk group belongs to $\Class_{\mathrm{CD}}(\mathrm{lin})$. 
We denote by $T=\{0, 1\}^*$ the rooted tree of finite binary words, where each word $w$ is connected by an edge to $wx, x\in \{0, 1\}$. Recall that the Grigorchuk group $G$ is the subgroup of the automorphism group $\aut(T)$ generated by the set of automorphisms $S=\{a, b, c, d\}$ given by the recursive rules
\[\begin{array}{lr}
a(0w)=1w, & a(1w)=0w;\\
b(0w)=0a(w), & b(1w)=0c(w);\\
c(0w)=0a(w), & c(1w)=1d(w);\\
d(0w)=0w, &d(1w)=1b(w).\\
\end{array}\]
We denote by $\Gamma_n$ the Schreier graph of the action of $G$ on the level $\{0, 1\}^n$ of the tree. We recall basic properties of the structure of these graphs, see  \cite{Bar-Gri-Hecke} for a more detailed description. Each $\Gamma_n$ is isometric to an interval of length $2^n-1$ in $\Z$ (with loops and multiple edges, which will not be important for us). Along these segments, binary words are ordered following the \emph{Gray code} ordering, namely each $w\in \{0, 1\}^n$ is connected by an edge in $\Gamma_n$ to the word obtained by changing its first digit, and to the word obtained by changing the digit that follows the first appearance of $0$ in $w$ (if the latter exists).  Thus all words have exactly two neighbours, with the exception of the words $11\cdots11$ and $11\cdots 10$, which lie at the extreme points of $\Gamma_n$ (we shall picture this with $11\cdots 11$ as the leftmost point and $11\cdots 10$ as the rightmost point). Let us denote by $d_n$ the associated distance on $\{0, 1\}^n$.
The map $p_n\colon \{0, 1\}^n\to \{0, 1\}^{n-1}$ that erases the last digit induces a covering map $p_n\colon \Gamma_n\to \Gamma_{n-1}$, which corresponds to folding $\Gamma_n$ around its middle edge.

\begin{prop} \label{p-Grigorchuk-linear-displacement}
	Let $G$ be the Grigorchuk group and retain the above notation. Then for every non-trivial element $g\in G$, there exists $n\ge 1$ and $w\in \{0, 1\}^n$ such that $d_n(gw, w)\ge \frac{1}{8} \diam(\Gamma_n)$. In particular $G$ belongs to $\Class_{\mathrm{CD}}(\mathrm{lin})$. 
\end{prop}

Recall that \textbf{portrait} of an element $g\in G$ is the collection permutations $(\sigma_w)_{w\in T}$, where each $\sigma_w\in \Sym(2)$ specifies the action of $g$ on the two children of $w$. In formulas, we have $g(wx)=g(w)\sigma_w(x)$ for $x\in \{0, 1\}$. We say that $w$ is an \textbf{active vertex} of $g$ if $\sigma_w$ is non-trivial. We will rely on a result in \cite{Arz-Sun}, where Arzhantseva and \v{S}uni\'c provide an explicit list of all the restrictions of portraits of elements of Grigorchuk groups to finite subtrees of depth 3 (and show that this characterises the closure of $G$ in $\aut(T)$). 

\begin{proof}[Proof of Proposition \ref{p-Grigorchuk-linear-displacement}]
	Let $g\in G$ be non-trivial.  Let $m$ be the smallest level containing active vertices. We can suppose that $G$ fixes $\{0, 1\}^n$ for all $n\le 3$, or the conclusion holds true trivially for $g$, since $\diam(\Gamma_n)\le 7$ for $n\le 3$ and any point which is not fixed by $g$ is moved at distance at least 1. Hence $m \ge 3$. Choose $w\in \{0, 1\}^m$ active, let $w_0$ be the projection of $w$ at level $m-3$ (i.e., the word obtained by removing the last three digits from $w$), and $T_0$ be the finite rooted at $w_0$ and containing all its descendants up to 3 levels. Hence $w$ is a leaf of $T_0$, and all active vertices for $g$ on  $T_0$ must be leaves, by minimality of $m$. It follows from Theorem 1 in \cite{Arz-Sun} that at least another leaf of $T_0$, distinct from $w$, must be active. It follows that we can find two distinct active vertices $w, v\in \{0, 1\}^m$ having the same projection at level $m-3$. Consider now the action on level $m+1$.  Since $g$ fixes level $m$, it must preserves all fibers of the covering map $p_{m+1}\colon \Gamma_{m+1}\to \Gamma_m$, which consist each of two points, and it acts non-trivially on each of the two fibers $p_{m+1}^{-1}(w), p_{m+1}^{-1}(v)$. Let $I\subset \Gamma_m$ be the interval ending at the right-most points of  length $2^{m-3}\ge\frac{1}{8}\diam(\Gamma_m)$. For any vertex $u\notin I$, the two points in $p_{m+1}^{-1}(u)$ are at distance at least $2|I|\ge\frac{1}{8}\diam(\Gamma_{m+1})$ in $\Gamma_{m+1}$. On the other hand the projection map from $ \Gamma_m\to \Gamma_{m-3}$ is injective on $I$, and since $w, v$ have the same image, they cannot both belong to $I$. Hence the desired conclusion follows, for $n=m+1$. \qedhere
\end{proof}

\subsection{The proof of Proposition \ref{prop-intro-grig}} Note that Proposition \ref{prop-intro-grig} from the introduction follows from Proposition \ref{prop-large-displ-stable}, together with the fact that all groups in the list belong to $\Class_{\mathrm{CD}}(\mathrm{lin})$, which has been established in Example \ref{e-LD}, Proposition \ref{prop-wreath-LD} and Proposition \ref{p-Grigorchuk-linear-displacement}.



\section{Schreier growth gaps for free products} \label{sec-gap-free-prod}

\subsection{A criterion} Let $G$ be a finitely generated group endowed with a finite symmetric generating set $S$. If $X$ is a $G$-set, we endow $X$ with  its Schreier graph structure. The associated simplicial distance is denoted $d_{G, X}$ (we adopt the convention that $d_{G, X}(x, y)=+\infty$ for points in distinct $G$-orbits). We denote by $B_{G,X}(x, n)$ the corresponding balls. The \textbf{support} of an element $g\in G$ in $X$ is the set $\supp_X(g)=\{x\in X\colon g x \neq x\}$. Given $R>0$, we shall say that two subsets $A, B\subset X$ are \textbf{$R$-separated} if $d_{G,X}(a, b)\ge R$ for every $a\in A$ and $b\in B$. 
Finally the \textbf{$R$-coarse connected components} of a subset $A\subset X$ are the equivalence classes of the equivalence relation on $A$ generated by the pairs $(a_1, a_2)$ such that $d_{G,X}(a_1, a_2)\le R$.

\begin{defin}
\label{def:C}
Let $G$ be a finitely generated group endowed with a finite symmetric generating set $S$. Let $\alpha, \beta \colon \mathbb{N} \to \mathbb{N}$ be functions such that $\alpha(n) \preccurlyeq \beta(n)$. We say that $G$ satisfies the\textbf{ sparse support condition at scale} $(\alpha, \beta)$,  if  for every faithful $G$-set $X$ such that $\vol_{G, X}(n)\not \succcurlyeq \beta(n)$, there exist constants $C, D>0$ such that for every $R>0$, there exists a non-trivial element $g\in G$ satisfying the following:
\begin{itemize}
\item[(C1)] We have $d_{G, X}(x, gx)\le D$ for every $x\in X$. 
\item[(C2)] Every $D$-coarse connected component of $\supp_X(g)$ has diameter at most $C$. 
\item[(C3)] Any two distinct $D$-coarse connected components of $\supp_X(g)$ are $R$-separated.
\item[(C4)] For every $x\in \supp_X(g)$, we have $|B_{G,X}(x, R)|\ge \frac{1}{C} \alpha(\frac{1}{C}R)$. 
\end{itemize}
\end{defin}
It is routine to check that this condition does not depend on the choice of a finite generating set $S$ for $G$ (using that a change of the generating set induces a bi-Lipschitz equivalence of the Schreier graphs, with constant depending on the generating sets only). Note that condition (C4) implies in particular that $G$ satisfies a Schreier growth gap $\alpha(n)$. The sparse support condition implies that this gap can be improved for any non-trivial free product with $G$, as follows. 
\begin{prop}\label{p-sparse-support}
Suppose that $G$ satisfies the sparse support condition at scale $(\alpha, \beta)$. Then for every non-trivial finitely generated group $H$, the group $G\ast H$ has a Schreier growth gap $\min(n \alpha(n), \beta(n))$.

\end{prop}

\begin{proof}
Set $L=G\ast H$. We fix a finite generating set of $L$ of the form $S\cup T$, where $S$ is a generating set of $G$ and $T$ is a generating set of $H$. Let $t\in T$ be a non-trivial generator of $H$. Let $X$ be a faithful $L$-set. We shall consider on $X$ the Schreier graph distance by $d_{L, X}$, as well as the distance $d_{G, X}$ induced by restricting the action to $G$. Suppose that $\vol_{G, X}(n)\not \succcurlyeq \beta(n)$ (else, the desired conclusion is true trivially). Let $C, D$ be as in Definition \ref{def:C}, fix $R>0$ (which we may assume is even), and let $g\in G$ be the corresponding element. Consider the commutator $h=[g, t]$. Note that $\supp_X(h)$ is contained in the 1-neighbourhood of $\supp_X(g)$ (with respect to the distance $d_{L, X}$). 

Since $h$ has infinite order and the action is faithful,  we can find $x_0\in X$ such that the points $x_n=h^n(x_0)$ are pairwise distinct for $n=0,\ldots, R$. Note that $d_{L, X}(x_n, x_{n+1})\le 2D+2$, by condition (C1). For each point $x_n$ we  choose $y_n \in \supp_X(g)$ such that $d_{L, X}(x_n, y_n)\le 1$. Then all points $y_0,\ldots, y_R$ belong to the ball $B_{L, X}(x_0, D_1R)$, with $D_1=2D+3$. Since all $D$-coarse connected components of $\supp_X(g)$ with respect the distance $d_{G, X}$ have diameter bounded by $C$, their 1-neighbourhood in the distance $d_{L, X}$ has cardinality uniformly bounded by some constant $C_1>0$  (not depending on $R$).

Now from the fact that the points $x_n$ are all distinct, we deduce that at most  $C_1$ points  $y_0, \ldots, y_R$ belong to the same $D$-coarse connected component. It follows that from $\{y_0,\ldots, y_n\}$ we can extract a collection of points $z_{1}, \ldots, z_{\lfloor R/ C_1 \rfloor}$ belonging to distinct $D$-coarse connected components. Then the balls $B_{G, X}(z_i, R/2)$ are pairwise disjoint by (C3), moreover they are contained in $B_{L, X}(x_0, D_2R)$, with $D_2=D_1+1/2$. From (C4) we deduce that 
\[|B_{L, X}(x_0, D_2R)|\ge \sum_{i=1}^{\lfloor R/C_1\rfloor} |B_{G, X}(z_i, R/2)|\ge \frac{1}{C_2} R\alpha (\frac{1}{C_2} R),\]
 for some constant $C_2>0$, which finishes the proof.  \qedhere
\end{proof}

\begin{remark} \label{r-sparse-support}
We point out that in the proof of Proposition \ref{p-sparse-support}, the only property of the free product $L=G\ast H$ that has been used is the following: $L$ is an overgroup of $G$ containing an element $t\in L$ such that for every non-trivial $g\in G$, the commutator $[t, g]$ has infinite order.  
\end{remark}

\subsection{Confined subgroups}

We recall the following definition.

\begin{defin}
Let $G$ be a group. A subgroup $H$ of $G$ is \textbf{confined} if there exists a finite subset $P \subset G \setminus \{1\}$ such that $H^g \cap P \neq \emptyset$ for all $g \in G$.
\end{defin}

Equivalently, a subgroup $H$ is confined if the closure of the set of conjugates of $H$  in the space $\sub(G)$ of subgroups of $G$, does not contain the trivial subgroup. The following simple lemma explains the usefulness of the notion of confined subgroups for the study of growth of actions. See \cite[Lemma 1.8]{LB-MB-solv} for a proof. If $X$ is a $G$-set and $x \in X$, we denote by $G_x$ the stabilizer of $x$ in $G$.


\begin{lem}\label{lem:SmallGrowthConfined}
Let $G$ be a finitely generated group, and $X$ a $G$-set. Let \[ \mathcal{S}(X) = \overline{ \left\lbrace G_x  \, : \, x \in X \right\rbrace } \subseteq \sub(G). \] Then for every $H \in \mathcal{S}(X)$, we have $\vol_{G, G/H}(n) \preccurlyeq \vol_{G, X}(n)$. In particular if some $G_x$ is not confined then $\vol_{G,X}(n) \simeq \vol_G(n)$. 
\end{lem}

Here $ \vol_G(n)$ represents the standard word growth given by the size of an $n$-ball in the Cayley graph of $G$ with respect some finite generating subset.




We will invoke the following, which is a particular case of Proposition 1.6 in \cite{LB-MB-solv}. 

\begin{lem}
	\label{lem:FiniteIndex}
	Let $G$ be a finitely-generated group, and let $H \leq K$ be subgroups of $G$. Then $\vol_{G, G/K} \le \vol_{G, G/H}$, and if in addition $H$ has finite index in $K$, then $\vol_{G, G/H} \simeq \vol_{G, G/K}$.
\end{lem}



\subsection{Houghton groups} \label{subsec-houghton}

If $\Omega$ is a set, $\Sym(\Omega)$ is the group of permutations of $\Omega$. We denote by $\Alt_{\mathrm{f}}(\Omega)$  the group of alternating finitely supported permutations of $\Omega$. For a subgroup $H \leq \Sym(\Omega)$, write $\Omega_{H,\mathrm{f}}$ for the union of finite orbits under $H$, and $\Omega_{H,\infty}$ for the union of infinite ones. We say that a partition is $H$-invariant if every element of $H$ sends a block of the partition to a (possibly different) block of the partition.

The following is a special case of \cite[Proposition 3.7]{LB-MB-ht}. 

\begin{prop}\label{p-confined-partially-finitary}
Let $\Omega$ be a countable set, and let $G \leq \Sym(\Omega)$ be any subgroup containing $\Alt_{\mathrm{f}}(\Omega)$.  Then a subgroup $H\le G$ is confined if and only if $\Omega_{H, \mathrm{f}}$ is finite and there exists an $H$-invariant partition $\Omega_{H,\infty} = \Omega_1 \cup \cdots \cup \Omega_k$ such that $H$ contains $\Alt_{\mathrm{f}}(\Omega_1) \times \cdots \times \Alt_{\mathrm{f}}(\Omega_k)$.
\end{prop}


For $r\ge 2$, let $\Xi_r$ be the graph obtained by gluing $r$ infinite rays $\ell_1, \cdots, \ell_r$ at their origin. The Houghton group $H_r$ is the group of all permutations $g$ of the vertex set of $\Xi_r$, such that there exists finite subset $E, F\subset \Xi_r$ such that for every $i$,  $g|_{\ell_i\setminus E}$ is an isometry onto $\ell_i\setminus F$. 
The group $H_r$ is finitely generated for $r \ge 2$. Since the action of $H_r$ on $\Xi_r$ is by permutations of bounded displacement, it satisfies $\vol_{H_r, \Xi_r}(n)\simeq n$. 

In what follows, we fix $r\ge 2$. Let us denote $\Omega$ the vertex set of $\Xi_r$,  and by $A:=\Alt_{\mathrm{f}}(\Omega)$. Recall that $A$ is simple, and every non-trivial subgroup of $\Sym(\Omega)$ normalized by $A$ contains $A$. In particular $A$ is a normal subgroup of $H_r$. 


\begin{prop} \label{p-Houghton}
For $r \ge 2$, let $X$ be a faithful $H_r$-set such that $\vol_{G, X}(n) \not \succeq n^2$. Let $X=\sqcup_{i \in I} X_i$ be the decomposition of $X$ into  $H_r$-orbits, and let $J$ be the subset of $i\in I$ such that the $H_r$-action on $X_i$ is  faithful.
Then:
\begin{enumerate}
\item \label{i-A-trivial} for $i\notin J$, the $A$-action on $X_i$ is trivial;
\item \label{i-A-non-trivial} there exist finite  $H_r$-sets $(Y_i)_{i\in J}$ of bounded cardinality (with respect to $i$) such that for every $i\in J$, the $H_r$-set $X_i$ is isomorphic to the product $H_r$-set $\Omega\times Y_i$.

\end{enumerate}
In particular, if $Z=\bigcup_{i\notin J} X_i$ is the set of $A$-fixed points, there exists a finite index subgroup $K$ of $H_r$ such that every $K$-orbit in $X\setminus Z$ is isomorphic to $\Omega$ as a $K$-set.
\end{prop}

\begin{proof}

Set $G=H_r$. First note that item \eqref{i-A-trivial} follows from the fact that $A$ is contained in every non-trivial normal subgroup of $G$. We will show \eqref{i-A-non-trivial}, which is the content of the statement.
We fix a word metric $\lVert \cdot \rVert$ on $G$, associated to some symmetric generating set. Let $X$ be  a $G$-set as in the statement and $K$ be the stabiliser of a point $x_0\in X_i$, with $I\in J$. The group $G$ has exponential growth, so by Lemma~\ref{lem:SmallGrowthConfined}, $K$ must be confined. Thus, we can apply Proposition \ref{p-confined-partially-finitary}.

Let us first show that $|\Omega_{K, \mathrm{f}}|\le 1$. Suppose by contradiction that this is not the case. Let $K_0\le K$ be the pointwise stabilizer of the finite set $\Omega_{K, \mathrm{f}}$, which has finite index in $K$. Then $\vol_{G, G/K_0}(n)\simeq \vol_{G, X_i}(n)$ by Lemma~\ref{lem:FiniteIndex}.  Choose two distinct points $x_1, x_2\in \Omega_{K, \mathrm{f}}$ and let $K_1 \ge K_0$ be their pointwise stabiliser. Since $\vol_{G, X_i}(n) \succcurlyeq \vol_{G, G/K_1}(n)$, a contradiction will follow if we show that $\vol_{G, G/K_1}(n)\succcurlyeq n^2$. Since the action of $G$ on $\Omega$ is highly transitive (that is, transitive on $n$-tuples of distinct points), upon replacing $K_1$ by some $G$-conjugate we suppose that $x_i$ is the point at position 1 on $\ell_i$ for $i=1, 2$ (with the convention that the common origin is at position 0).

Now for $i=1, 2$, let  $\sigma_i$ be the transposition that swaps the first and second positions of $\ell_i$. Let also $t_i\in G$ be any element such that $t_i|_{\ell_i}$ shifts $\ell_i$ inside itself by 1. For example, these can be taken to be $t, t^{-1}$ where $t$ is the shift along a $\Z$-isomorphic ray coming from juxtaposing the rays $\ell_1, \ell_2$. For $i=1, 2$ and $n\ge 1$ consider the element
 \[\gamma_{i, n}=(t_i^n\sigma_i t_i^{-n})(t_i^{n-1} \sigma_i t^{-n+1}) \cdots \sigma_i=t_i^n\sigma_i (t_i^{-1}\sigma_i)^n.\]
Note that $\lVert\gamma_{i, n}\rVert \le Cn$, with $C=\max\{\lVert \sigma_i \rVert, \lVert t_i\rVert \;|\; i=1,2\}$.  On the other hand each $\gamma_{i, n}$ is a permutation with finite support contained in $\ell_i$, and $\gamma_{i, n}(x_i)$ is the point at position $n+1$ on $\ell_i$. Hence applying products of the form $\gamma_{1, m_1} \gamma_{2, m_2}$ to the pair $(x_1, x_2)$, with $0\le m_1, m_2\le n$ shows that $\vol_{G, G/K_1}(n)\succcurlyeq n^2$. This is a contradiction and proves the claim that $|\Omega_{K, \mathrm{f}}|\le 1$.

We now consider the partition $\Omega_{K, \infty}=\Omega_1\cup \cdots \cup\Omega_k$ and claim that $k=1$.  Suppose that $k\ge 2$. By a similar reasoning as above (passing first to a finite index subgroup of $K$, and then to an overgroup), we have $\vol_{G, X}(n)\succcurlyeq \vol_{G, G/K_1}(n)$, where $K_1$ is the subgroup of $G$ of elements that preserve both $\Omega_1$ and $\Omega_2$, so a contradiction will follow if we show that $\vol_{G, G/K_1}(n)\succcurlyeq n^2$. Fix $n$ large enough. Again by high transitivity  and using the fact that $\Omega_1, \Omega_2$ are both infinite, we can find $g\in G$ such that the point at position 1 on $\ell_1$ is in $g(\Omega_1)$, and all points on $\ell_1$ at positions between 2 and $n$ are in $g(\Omega_2)$, while the symmetric statement holds for $\ell_2$. Now applying to $g(\Omega_1)$ and $g(\Omega_2)$ the same the elements $h=\gamma_{1, m_1}\gamma_{2, m_2}$ defined above with $1\le m_i\le n$, we obtain $n^2$ different pairs $(hg(\Omega_1), hg(\Omega_2))$ with $\lVert h\lVert \le 2Cn$. This shows that the ball of radius $2Cn$ around $gK_1$ has cardinality at least $n^2$. Hence $k=1$. 

It follows that $K$ contains $\alt(\Omega_{K, \infty})$. This also shows that $\Omega_{K, \mathrm{f}}\neq \varnothing$, since otherwise $K$ would contain the normal subgroup $A$, contradicting that the action on $X_I$ is faithful. Hence $K\cap A$ contains a point stabiliser for the action of $A$ on $\Omega$, and hence is equal to it, as the latter is a maximal subgroup of $A$. Since $K$ was an arbitrary point stabiliser, we deduce that the action of $A$ on $X_i$ is conjugate on each of its orbits to the standard action of $A$ on $\Omega$. For each $A$-orbit $\mathcal{O}\subset X_i$ we denote by $j_\mathcal{O}\colon \mathcal{O}\to \Omega$ the unique $A$-equivariant bijection (its uniqueness follows from the fact that the only $A$-equivariant permutation of $\Omega$ is the identity). 

Let $Y_i=X_i/A$, the space of $A$-orbits, on which $G$ acts since $A$ is normal. For a point $x\in A$, we denote by $\mathcal{O}(x)$ its $A$-orbit. Consider the map $\iota\colon X_i \to \Omega \times Y_i$ given by $\iota(x)=(j_{\mathcal{O}(x)}(x), \mathcal{O}(x))$. This map is equivariant (this is obvious for the second component; for the first, it follows from the observation that the map $g^{-1}\circ j_{\mathcal{O}(gx)} \circ g$ is an $A$-equivariant bijection from $\mathcal{O}(x)$ to $\Omega$, and thus it is equal to $j_{\mathcal{O}(x)}$). The map $\iota$ is also clearly injective. To check that it is surjective it is enough to observe that the diagonal $G$-action on $\Omega \times Y_i$ is transitive: this is true because the $G$-action on $Y_i$ is transitive, and the group $A$ acts trivially on $Y_i$ and transitively on each fiber $\Omega\times \{y\}$.

To conclude we need to argue that  $Y_i$ must have uniformly bounded cardinality when $i\in J$ varies. Fix $n$ and suppose that the diameter of $Y_i$ is at least $n$. Then we can find a point $y\in Y_i$ and elements $s_1, \ldots, s_n$ in the generating set of $G$ such that, if we set $g_m=s_m\cdots s_1$, then the points $y, g_1y,  \ldots, g_ny$ are pairwise distinct (the point $y$ and  sequence $(s_i)$ can be found by looking at any path in the Schreier graph of $Y_i$ of length at least $n$). Next consider again the elements $\gamma_{i, n}$ above. Note that for every $n$, the element $\delta_n=\gamma_{1, n} \gamma_{2, n}$ belongs to $A$ and thus acts trivially on $Y_i$, and coincides with $\gamma_{1, n}$ in restriction to $\ell_1\subset \Omega$. Let $x\in \ell_1$ be at position 1, and consider the point $(x, y)\in \Omega \times Y_i$. Then the $n^2$ points $g_{m_2} \delta_{m_1}(x, y)=(g_{m_2}\gamma_{1,m_1}x, g_{m_2}y)$ for $m_1, m_2\le n$ are pairwise distinct, showing that there are at least $n^2$ in a ball of radius $O(n^2)$. This cannot be true for $n$ arbitrarily large, showing that the cardinality of $Y_i$ must be uniformly bounded. 

Finally, since the group $G$ is finitely generated, we may choose a finite index subgroup $K$ which acts trivially on $Y_i$ for every $i$, showing the last sentence in the statement. \qedhere
\end{proof}

\begin{thm} \label{t-Houghton}
For every $r\ge 2$ and every non-trivial finitely generated group $K$, the group $H_r\ast K$ has a Schreier growth gap $n^2$. 
\end{thm}

\begin{proof}
By Proposition \ref{p-sparse-support}, it is enough to show that $G=H_r$ satisfies the sparse support condition at scale $(n, n^2)$. This is an immediate consequence of Proposition \ref{p-Houghton}. Indeed fix any non-trivial element $g\in A$. Then  for every faithful $G$ set $X$ such that $\vol_{G, X}(n) \not \succcurlyeq n^2$, the element $g$ has finite support in each $G$-orbit, so one can find constants $C, D>0$ such that $g$ satisfies conditions (C1)--(C3) for every $R>0$, and condition (C4) is automatic, because every infinite connected graph has at least linear growth. \qedhere
\end{proof}

\begin{remark}
We point out that the proof of Theorem  \ref{t-Houghton}, the assumption that the group $L:=H_r\ast K$ is a free product is only used to invoke  Proposition \ref{p-sparse-support}. As a consequence, this assumption can be relaxed as in Remark \ref{r-sparse-support}. 
 \end{remark}
 
 \begin{remark}
To prove Theorem \ref{t-Houghton} we could have restricted to the case $r=2$, since every $H_r$ contains $H_2$ as a subgroup. Note however that the previous proof provides additional information on actions of small growth of $H_r$ for general $r$ (Proposition \ref{p-Houghton}). 
\end{remark}

\subsection{Topological full groups} \label{subsec-full-group}

Throughout this section we let $Z$ be a Cantor space, and $\mathcal{G}\le \homeo(Z)$ be a subgroup of its group of homeomorphisms. Recall that the \textbf{topological full group} of $\mathcal{G}$ is the group $\Fsf(\mathcal{G})$ of all homeomorphisms $h$ of $Z$ such that for every $z\in Z$, there exist a clopen neighbourhood  $U$ of $z$ and $g\in \mathcal{G}$ such that $h|_U=g|_U$. When  $\mathcal{G}$ is a cyclic group generated by a single homeomorphism $\varphi$, we write $\Fsf(\varphi)$ instead of $\Fsf(\mathcal{G})$. The group $\Fsf(\mathcal{G})$ is countable provided $\Gcal$ is countable.

We say that an element $g\in \Fsf(\Gcal)$ is a \textbf{3-cycle} if $g^{3}=1$ 
and its support $\supp(g):=\{z\colon g(z)\neq z\}$ is clopen and can be partitioned into 3 distinct clopen subsets $U_1, U_2, U_3$ such that $g(U_i)=U_{i+1}$ (with $i$ mod 3). Following Nekrashevych \cite{Nek-simple}, the subgroup of $\Fsf(\Gcal)$ generated by 3-cycles is called the \textbf{alternating full group} of $\Gcal$, and is denoted $\Asf(\Gcal)$. It is shown in \cite{Nek-simple}  that if the action of $\Gcal$ on $Z$ is \textbf{minimal} (that is, all its orbits are dense), then the group $\Asf(\Gcal)$ is simple and contained in every non-trivial normal subgroup of $\Fsf(\Gcal)$. It is also shown in \cite{Nek-simple} that if $\Gcal$ is finitely generated and its action on $Z$ is expansive and all its  orbits have cardinality at least 5, then $\Asf(\Gcal)$ is finitely generated. Recall that an action of a finitely generated group $\Gcal$ on the Cantor space $Z$ is \textbf{expansive} if and only if it is conjugate to a \textbf{subshift} (a $\Gcal$-invariant subset of $\{1, \cdots d\}^\Gcal$ for some $d\ge 2$).
In many cases, results of Matui show that $\Asf(\Gcal)$ coincides with the commutator subgroup of $\Fsf(\Gcal)$ (see his survey \cite{Mat-survey}).  

It is well-known (and not difficult to see) that if the group $\Gcal$ is finitely generated, then for every finitely generated subgroup $G$ of $\Fsf(\Gcal)$, the identity map on $Z$  defines a Lipschitz map between the graphs of the actions $\Gamma(G, Z)\to \Gamma(\Gcal, Z)$. Since the groups $\Asf(\Gcal)$ and $\Gcal$ have the same topological full groups, it follows that if the action of $\Gcal$ on $Z$ is expansive and has no orbits of cardinality less than 5, then the graph $\Gamma(A(\Gcal), Z)$ is bi-Lipschitz equivalent to $\Gamma(\Gcal, Z)$. In particular $\vol_{\A(\Gcal), Z}(n)\simeq \vol_{\Gcal, Z}(n)$. 

For a group $G$ acting on $Z$ by homeomorphisms, and for a finite set $Q\subset  Z$, we denote by $G_Q$ the setwise stabiliser of $Q$ and by $G^0_Q$  its subgroup consisting of elements fix pointwise some neighbourhood of every $z\in Q$. The following is proven in \cite{MB-full}.

\begin{thm}[\cite{MB-full}] \label{t-confined-full}
Let $\Gcal\le \homeo(Z)$ be a group acting minimally on $Z$, and let $G=\Asf(\Gcal)$. Then a subgroup $H$ of $ G$ is confined if and only if there exists a finite set $Q\subset Z$ such that $G^0_Q\le H \le G_Q$ (the set $Q$ is moreover unique).
\end{thm} 

\begin{cor} \label{c-full-small-growth}
Retain the assumptions of Theorem \ref{t-confined-full}, and assume further that $\Gcal$ is finitely generated and that its action on $Z$ is expansive. Let  $\alpha(n)=\vol_{\mathcal{G}, Z}(n)$. Let $X$ be a faithful $G$-set such that $\vol_{G, X}(n)\not \succcurlyeq \alpha(n)^2$, and let $H$ be the stabiliser of a point in $X$. Then there exists $z\in Z$ such that $G^0_z\le H\le G_z$. 
\end{cor}

\begin{proof}
The assumption implies that $H$ must be confined (Lemma \ref{lem:SmallGrowthConfined}). Let $Q\subset Z$ be the finite subset given by Theorem \ref{t-confined-full}. Enumerate the points of $Q$ as  $z_1,\ldots, z_r$ and choose pairwise disjoint clopen neighbourhoods $U_i$ of $z_i$. For each $i$ we choose a finitely generated subgroup $G_i\le G$, supported on $U_i$, such that the inclusion map $U_i\to Z$ defines a quasi-isometry of graphs $\Gamma(G_i, U_i)\to \Gamma(G, Z)$ (one can choose $G_i$ to be the alternating full group of the restriction of the groupoid of germs of $\mathcal{G}$ to $U_i$; its finite generation follows from \cite[Proposition 5.4]{Nek-simple} and the quasi-isometry of the graphs from  \cite[Cor 2.3.4]{Nek-hyp}). 
It follows that the action of $G_1\times \cdots \times G_r$ on the orbit of $Q$ has growth bounded below by $\alpha(n)^r$, from which it follows that $\vol_{G, G/H}(n)\ge \vol_{G, G/G_Q}(n)\succcurlyeq \alpha(n)^r$, which contradicts the assumption unless $r=1$. \qedhere
\end{proof}


Recall that for any subgroup $G \le \homeo(Z)$, the maps $Z \to \sub(G)$ defined by $z\mapsto G_z$ and $z\mapsto G^0_z$ are respectively upper and lower semi-continuous. When $G$ is countable, a Baire argument implies that $G_z=G^0_z$ for $z$ in a dense $G_\delta$ subset.

\begin{thm}
Assume that $\Gcal$ is finitely generated and its action on $Z$ is minimal and expansive, and let  $\alpha(n)=\vol_{\mathcal{G}, Z}(n)$.  Then for every non-trivial finitely generated group $H$, the group $G = \Asf(\Gcal)\ast H$ has a Schreier growth gap $n\alpha(n)$. 
\end{thm}
\begin{proof}
We shall show that $G$ satisfies the sparse support condition at scale $(\alpha(n), \alpha(n)^2)$. As usual, we fix once and for all a finite symmetric generating set $S$ for $G$, which is used to define all Schreier graphs and distances below. Let $X$ be a faithful $G$-set with $\vol_{G, X}(n)\not \succcurlyeq \alpha(n)^2$. We shall show that it  satisfies the conclusion of Definition \ref{def:C} with constants $D=C=2$. To this end, fix $R>2$. 
 
 By \cite[Lemma 6.2]{LB-MB-comm-lemma}, for every $G$-orbit $\mathcal{O}$ in $Z$, we have $\vol_{G, \mathcal{O}}(n)=\alpha(n)$. Hence we can find a point $z_0\in Z$ satisfying the generic condition that $G^0_{z_0}=G_{z_0}$ and such that $|B_Z(z_0, R)|=\alpha(R)$.  Choose $s_1, s_2\in S$ such that $z_0, s_1(z_0),s_2s_1(z_0)$ are distinct. Let $(V_{n})$ be a system of clopen neighbourhoods of $z_0$, and such that $V_n, s_1(V_n), s_2s_1(V_n)$ are disjoint for every $n$. Define accordingly the sequence $(h_n)$ of 3-cycles given by 
 
 \[h_n(z)=\left\{ \begin{array}{lc} s_1(z) & z\in V_n\\ s_2(z) & z\in s_1(V_n) \\ s_1^{-1}s_2^{-1}(z) & z\in s_1s_1(V_n)\\ z & \text{ else.}
 
 \end{array}\right.
  \] 
  Observe that $d_{G, Z}(z, h_n(z))\le 2$ for every $z\in Z$,  by construction. 
  The supports $\supp(h_n)$  decreases to $\{z_0, h_n(z_0), h_n^2(z_0)\}$ as $n\to \infty$. It follows that for  $n$ is sufficiently large, we have $B_Z(z_0, R) \cap \supp_Z(h_n)= \{z_0, s_1(z_0), s_2s_2(z_0)\}$. 
  This shows that for $n$ large, the element $g=h_n$ satisfies the conditions in Definition \ref{def:C} for the action on $Z$. 
  
  Now consider the action on $X$ and pick a point $x\in X$, and denote by $\phi(x)\in Z$ the unique point such that $G^0_{\phi(x)}\le G_x \le G_{\phi(x)}$, given by Corollary \ref{c-full-small-growth}. Note that $\supp_X(h_n)\subset  \phi^{-1}(V_n\cup s_1(V_n)\cup s_2s_1(V_n))$, since $h_n\in G^0_{\phi(x)}$ for $\phi(x)\notin V_n\cup s_1(V_n)\cup s_2s_1(V_n)$. However since the point $z_0$ satisfies $G^0_{z_0}=G_{z_0}$ (and so do $s_1(z_0), s_2s_1(z_0)$) then it follows from semi-continuity  and the inclusions $G^0_\phi(x)\le G_x \le G_{\phi(x)}$ that for $n$ large enough every $x$ such that $\phi(x)\in  V_n$, then the ball $B_X(x, R)$ will be isomorphic to the corresponding ball  $B_Z( z_0, R)$, and similarly if $\phi(x)\in s_1(V_n), s_2s_1(V_n)$, with $z_0$ replaced by the corresponding image. Therefore the element $g=h_n$ satisfies all items in Definition \ref{def:C} for the action on $X$ as well. 
  \qedhere
 \end{proof}

\bibliographystyle{amsalpha}
\bibliography{bib-growth-free-prod}

\end{document}